\documentclass[12pt,oneside]{amsart}

\pdfoutput=1
\usepackage{xcolor}

\usepackage[section]{placeins}
\usepackage[margin=1in]{geometry}
\usepackage{mathabx}
\usepackage{xcolor}
\usepackage{graphicx}
\usepackage{amssymb}
\usepackage{tikz}
\usepackage{latexsym}
\usepackage{epsfig}
\usepackage{enumitem}
\usepackage{multicol}
\usepackage{wasysym}
\usepackage{amsmath}
\usepackage{amsthm}
\usepackage{amscd}
\usepackage{ulem}
\usepackage{soul}
\usepackage{multirow}
\usepackage{rotating}
\usepackage{pdflscape}
\usepackage{float}
\usepackage{dsfont}
\usepackage{comment}
\usepackage{booktabs}

\usepackage[colorlinks,citecolor=red,pagebackref,hypertexnames=false]{hyperref}
\hypersetup{
  colorlinks = true,
  citecolor=red,
  linkcolor  = red
}

\setlist{itemsep=.06125in}

\restylefloat{table}
\numberwithin{equation}{section}

\theoremstyle{plain}
\newtheorem{theorem}{Theorem}[section]
\newtheorem{lemma}[theorem]{Lemma}

\newtheorem{proposition}[theorem]{Proposition}

\theoremstyle{definition}

\theoremstyle{remark}
\newtheorem{remark}[theorem]{Remark}

\date{\today}

\author{W. Burstein, A. Iosevich, and A. Sant}
\address{Department of Mathematics, University of Rochester, Rochester, NY, USA}
\email{willburst88@gmail.com}
\address{Department of Mathematics, University of Rochester, Rochester, NY, USA}
\email{iosevich@gmail.com}
\address{Department of Mathematics, University of Rochester, Rochester, NY, USA}
\email{asant2@ur.rochester.edu}

\thanks{A. I. was supported in part by the National Science Foundation under NSF DMS-2154232.}

\title{Elliptic curves, Fourier ratio, and sampling complexity}

\begin{document}

\begin{abstract}
We study the normalized Frobenius trace associated with the Legendre family of elliptic curves over $\mathbb F_p$ from the point of view of Fourier complexity. If
\[
f(t)=\frac{a_p(E_t)}{\sqrt p},
\qquad
E_t:\ y^2=x(x-1)(x-t),
\]
with $f(0)=f(1)=0$, then
\[
\frac{\|\widehat f\|_1}{\|\widehat f\|_2}\asymp \sqrt p.
\]
More precisely, the Fourier transform of $f$ has squared $\ell^2$ norm of order $p$ while its individual coefficients remain uniformly bounded. It follows that no Fourier model supported on fewer than a sufficiently small constant multiple of $p$ frequencies can approximate $f$ in $\ell^2$ with error smaller than a fixed proportion of $\|f\|_2$.

We also show that the Fourier magnitude profile of $f$ supports a family of at least $\exp(cp)$ real-valued functions with identical Fourier magnitudes and identical Fourier ratio, any two of which are separated by at least $c\sqrt p$ in $\ell^2$. Consequently, every deterministic reconstruction procedure that recovers all members of this family from bounded-precision point evaluations must use at least $c_Bp$ samples, where $c_B>0$ depends only on the number of bits used to encode each observation. The arithmetic input is unconditional and relies only on the Weil bound for mixed character sums, the evaluation of the quadratic Gauss sum, and elementary character identities.
\end{abstract}



\maketitle

\section{Introduction}

The purpose of this paper is to study the complexity of arithmetic data arising from elliptic curves over finite fields from the point of view of Fourier analysis. In particular, we consider the Frobenius traces associated with one-parameter families of elliptic curves and investigate their spectral complexity.

Let $p\ge 5$ be a prime, and consider the Legendre family of elliptic curves
$$
E_t:\ y^2 = x(x-1)(x-t), \qquad t \in {\mathbb F}_p \setminus \{0,1\}.
$$

Associated to each $E_t$ is the Frobenius trace $a_p(E_t)$ \cite{SilvermanAEC}, which measures the deviation of the number of points on $E_t$ from the expected value $p+1$. We define
$$
f(t)=\frac{a_p(E_t)}{\sqrt p}
\qquad\text{for }t\in\mathbb F_p\setminus\{0,1\},
$$
and extend $f$ to all of $\mathbb F_p$ by setting
$$
f(0)=f(1)=0.
$$
By Hasse's bound, $|f(t)|\le 2$, so $f$ is a bounded real-valued function on $\mathbb F_p$.

The central question we address is the following: to what extent can the function $f$ be approximated by low-complexity spectral models? More concretely, how many Fourier coefficients are needed to approximate $f$ accurately in $L^2$?

Our approach is based on the Fourier Ratio, defined as the ratio of the $\ell^1$ norm to the $\ell^2$ norm of the Fourier transform of a function. This quantity serves as a measure of spectral complexity. Functions with small Fourier Ratio exhibit concentration in frequency space and are therefore easier to approximate, while functions with large Fourier Ratio are spread out and exhibit pseudorandom behavior.

The Fourier Ratio has recently emerged as a useful tool in a variety of settings, including signal complexity, spectral synthesis, approximation theory, recovery, and sampling complexity \cite{Aldaleh2025, BIKN2026, ILPY25, IMW2026, DI26}. In this paper we show that the normalized Frobenius trace associated with the Legendre family has large Fourier Ratio. More precisely, we prove that
\[
\frac{\|\widehat{f}\|_1}{\|\widehat{f}\|_2} \asymp \sqrt{p},
\]
i.e., the Fourier mass of $f$ is spread across $\asymp p$ frequencies at essentially bounded size.

This result reflects the well-known phenomenon that Frobenius traces in families of elliptic curves exhibit strong cancellation and behave in many respects like random sequences. The Sato--Tate law describes the limiting distribution of normalized traces, and Michel established corresponding results for families of elliptic curves \cite{Michel}. The square-root cancellation used below, however, is the classical Weil bound for mixed character sums; see, for example, \cite{IK}. The result obtained here can therefore be viewed as a discrete manifestation of this pseudorandom behavior in the language of Fourier complexity.

From the more general point of view of trace functions, the uniform boundedness of the normalized Fourier transform is consistent with the bounded-conductor theory developed in arithmetic geometry; see, for example, \cite{FKM}. The contribution of the present paper is not merely the boundedness of the individual Fourier coefficients, but the explicit combination of that bound with an exact second-moment calculation and the resulting consequences for Fourier ratio, approximation, metric complexity, and bounded-precision sampling.

From the point of view of approximation theory, a large Fourier Ratio indicates that the function $f$ has high spectral complexity. In particular, $f$ cannot be efficiently approximated by sparse Fourier models. We prove this in the strongest natural $L^2$ form: every Fourier polynomial supported on fewer than a sufficiently small constant multiple of $p$ frequencies remains a fixed positive distance from $f$.

More significantly, the Fourier mass of $f$ is sufficiently spread out that one can construct an exponentially large family of real-valued functions with identical Fourier magnitude profile and identical Fourier ratio. These functions are pairwise separated by a constant multiple of $\sqrt p$ in $\ell^2$. Thus a single Fourier magnitude profile is compatible with exponentially many well-separated signals.

As a consequence, any deterministic reconstruction procedure that is required to recover every member of this family from point evaluations, with a fixed number of bits used to encode each observed value, requires at least $cp$ samples for uniform recovery to error smaller than a fixed constant multiple of $\sqrt p$.

Our main result is the following.

\begin{theorem}[Main Result]\label{thm:main}
Let $p\ge 5$ be a prime and let
$$
f(t)=\frac{a_p(E_t)}{\sqrt p},
\qquad
E_t:\ y^2=x(x-1)(x-t),
$$
with $f(0)=f(1)=0$. Then the following hold:

\begin{enumerate}
\item The Fourier ratio satisfies
$$
FR(f)=\frac{\|\widehat f\|_1}{\|\widehat f\|_2}\asymp \sqrt p.
$$

\item There exist absolute constants $c,c'>0$ such that, whenever $\Lambda\subset\mathbb F_p$ satisfies $|\Lambda|\le c'p$, every Fourier polynomial
$$
T(t)=p^{-1/2}\sum_{m\in\Lambda}c_m\chi(mt)
$$
satisfies
$$
\|f-T\|_2\ge c\|f\|_2.
$$

\item There exists a family $\mathcal F_f$ of real-valued functions on $\mathbb F_p$ such that
$$
|\mathcal F_f|\ge \exp(cp),
$$
every $g\in\mathcal F_f$ satisfies
$$
|\widehat g(m)|=|\widehat f(m)|
\qquad\text{for all }m\in\mathbb F_p,
$$
and any two distinct $g,h\in\mathcal F_f$ satisfy
$$
\|g-h\|_2\ge c\sqrt p.
$$
In particular, every member of $\mathcal F_f$ has the same Fourier ratio as $f$.

\item Fix a positive integer $B$. Suppose that a deterministic reconstruction procedure is required to reconstruct every function in $\mathcal F_f$ to error less than $c\sqrt p/3$ in $\ell^2$, using $m$ point evaluations chosen either nonadaptively or adaptively, with each observed value encoded using at most $B$ bits. Then
$$
m\ge c_Bp,
$$
where $c_B>0$ depends only on $B$.

\end{enumerate}
\end{theorem}

The first two conclusions of the theorem concern the normalized Frobenius trace itself. The final two conclusions concern a larger signal class generated by its Fourier magnitude profile. The members of $\mathcal F_f$ need not themselves arise as Frobenius traces of elliptic curves, and they need not satisfy the pointwise Hasse bound. The role of the Legendre family is to provide an explicit arithmetic signal whose Fourier magnitude profile is sufficiently flat to support an exponentially large collection of pairwise separated signals. Accordingly, the sampling lower bound should be understood as a complexity consequence of this arithmetic Fourier profile, rather than as a lower bound for recovering the single function $f$.

The theorem shows that the normalized Frobenius trace has maximal-order Fourier ratio. In particular, any sparse Fourier approximation using fewer than $c'p$ frequencies must incur $L^2$ error at least a constant fraction of $\|f\|_2$.

More significantly, the function $f$ generates an exponentially large family of signals with identical Fourier magnitudes and identical Fourier ratio that are separated in $\ell^2$. This means that there are exponentially many distinct signals consistent with the same spectral magnitude data.

A sampling scheme based on $m$ point evaluations with bounded precision can distinguish at most $\exp(Cm)$ different signals. Since the construction above produces $\exp(cp)$ pairwise separated signals, accurate reconstruction requires
$$
\exp(Cm) \ge \exp(cp),
$$
which forces $m \ge cp$.

This paper is organized as follows. In Section~\ref{section:FR}, we compute the Fourier Ratio of the normalized trace $f$. In Section~\ref{section:VC}, we prove sparse approximation lower bounds and establish the local complexity result (Proposition~\ref{prop:local-family}). Section~\ref{section:proof} assembles these results to prove Theorem~\ref{thm:main}. In Section~\ref{section:sampling}, we derive the sampling complexity consequence. We conclude in Section~\ref{section:conclusion} with a discussion of generalizations and open problems.

\section{Fourier Ratio for the Legendre Family}
\label{section:FR}

Fix a prime $p\ge 5$. We write ${\mathbb F}_p$ for the field with $p$ elements. Let $\phi$ denote the quadratic character on ${\mathbb F}_p^\times$, defined by
$$
\phi(u)=
\begin{cases}
1 & \text{if $u$ is a nonzero square in ${\mathbb F}_p$}, \\
-1 & \text{if $u$ is a non-square in ${\mathbb F}_p$},
\end{cases}
$$
and extended to ${\mathbb F}_p$ by setting $\phi(0)=0$. Let
$$
\chi(u)=e^{\frac{2\pi i u}{p}},
\qquad u\in {\mathbb F}_p,
$$
be the standard additive character on ${\mathbb F}_p$.

For $t\in {\mathbb F}_p\setminus\{0,1\}$, consider the elliptic curve
$$
E_t:\ y^2=x(x-1)(x-t),
$$
and define
$$
f(t)=\frac{a_p(E_t)}{\sqrt p}.
$$
We extend $f$ to all of ${\mathbb F}_p$ by setting
$$
f(0)=f(1)=0.
$$

For a function $F:{\mathbb F}_p\to {\mathbb C}$, we define its normalized Fourier transform by
$$
\widehat{F}(m)=p^{-1/2}\sum_{x\in {\mathbb F}_p} F(x)\chi(-mx),
\qquad m\in {\mathbb F}_p.
$$
We write
$$
\|F\|_1=\sum_{x\in{\mathbb F}_p}|F(x)|,
\qquad
\|F\|_2=\left(\sum_{x\in{\mathbb F}_p}|F(x)|^2\right)^{1/2}.
$$
Finally, if $A$ and $B$ are nonnegative quantities, we write
$$
A\ll B
$$
or equivalently
$$
A=O(B)
$$
to mean that there exists an absolute constant $C>0$ such that
$$
A\le CB.
$$
We write
$$
A\asymp B
$$
to mean that
$$
A\ll B \quad \text{and} \quad B\ll A.
$$
By an absolute constant we mean a constant independent of $p$.

The main result of this section is the following.

\begin{theorem}\label{thm:LegendreFourierRatio}
Let $p\ge 5$ be a prime. There exist absolute constants $c,C>0$ such that
$$
c\sqrt p
\le
\frac{\|\widehat f\|_1}{\|\widehat f\|_2}
\le
C\sqrt p.
$$
More precisely,
$$
\|\widehat f\|_2^2=p-2-\frac{3}{p},
$$
and there is an absolute constant $C_0>0$ such that
$$
\|\widehat f\|_\infty\le C_0.
$$
In particular,
$$
\frac{\|\widehat f\|_1}{\|\widehat f\|_2}
\asymp
\sqrt p.
$$
\end{theorem}

\begin{proof}
For $t\in {\mathbb F}_p\setminus\{0,1\}$, one has the standard character-sum formula
$$
a_p(E_t)=-\sum_{x\in {\mathbb F}_p}\phi(x(x-1)(x-t)).
$$
Define
$$
g(t):=-p^{-1/2}\sum_{x\in {\mathbb F}_p}\phi(x(x-1)(x-t)),
\qquad t\in {\mathbb F}_p.
$$
Then
$$
g(t)=f(t)
\qquad \text{for } t\neq 0,1.
$$

Set
$$
h=f-g.
$$
Then $h$ is supported on the set $\{0,1\}$.

We first compute the exceptional values exactly. At $t=0$,
$$
g(0)
=
-p^{-1/2}\sum_{x\in\mathbb F_p}\phi(x^2(x-1))
=
-p^{-1/2}\sum_{x\ne 0}\phi(x-1)
=
\frac{\phi(-1)}{\sqrt p}.
$$
Similarly,
$$
g(1)
=
-p^{-1/2}\sum_{x\in\mathbb F_p}\phi(x(x-1)^2)
=
-p^{-1/2}\sum_{x\ne 1}\phi(x)
=
\frac{1}{\sqrt p}.
$$
Since $f(0)=f(1)=0$, it follows that
$$
h(0)=-\frac{\phi(-1)}{\sqrt p},
\qquad
h(1)=-\frac{1}{\sqrt p}.
$$
Consequently,
$$
\|h\|_2=\left(\frac{2}{p}\right)^{1/2}.
$$
Moreover, for every $m\in\mathbb F_p$,
$$
|\widehat h(m)|
\le
p^{-1/2}\bigl(|h(0)|+|h(1)|\bigr)
=
\frac{2}{p}.
$$
Hence
$$
\|\widehat h\|_1\le 2
\qquad\text{and}\qquad
\|\widehat h\|_2=\|h\|_2=\left(\frac{2}{p}\right)^{1/2}.
$$

We now compute $\widehat g(m)$. By definition,
\begin{align*}
\widehat g(m)
&=
-p^{-1}\sum_{t\in{\mathbb F}_p}\sum_{x\in{\mathbb F}_p}
\phi(x(x-1)(x-t))\chi(-mt) \\
&=
-p^{-1}\sum_{x\in{\mathbb F}_p}\phi(x(x-1))
\sum_{t\in{\mathbb F}_p}\phi(x-t)\chi(-mt).
\end{align*}
If $m=0$, then
$$
\sum_{t\in{\mathbb F}_p}\phi(x-t)=\sum_{u\in{\mathbb F}_p}\phi(u)=0,
$$
so
$$
\widehat g(0)=0.
$$

Suppose now that $m\neq 0$. Making the change of variables $u=x-t$, we obtain
\begin{align*}
\sum_{t\in{\mathbb F}_p}\phi(x-t)\chi(-mt)
&=
\sum_{u\in{\mathbb F}_p}\phi(u)\chi(-m(x-u)) \\
&=
\chi(-mx)\sum_{u\in{\mathbb F}_p}\phi(u)\chi(mu).
\end{align*}
Let
$$
\tau(\phi)=\sum_{u\in{\mathbb F}_p}\phi(u)\chi(u)
$$
be the quadratic Gauss sum. Then
$$
\sum_{u\in{\mathbb F}_p}\phi(u)\chi(mu)=\phi(m)\tau(\phi),
$$
and therefore
\begin{equation}\label{eq:ghatformula}
\widehat g(m)
=
-\frac{\tau(\phi)}{p}\phi(m)
\sum_{x\in{\mathbb F}_p}\phi(x(x-1))\chi(-mx),
\qquad m\neq 0.
\end{equation}
Since
$$
|\tau(\phi)|=\sqrt p,
$$
equation~\eqref{eq:ghatformula} and the Weil bound for mixed character sums give
$$
|\widehat g(m)|
\le
p^{-1/2}\cdot 2\sqrt p
=
2
$$
for every $m\neq 0$. Since $\widehat g(0)=0$, it follows that
\begin{equation}\label{eq:Linftybound}
|\widehat g(m)|\le 2
\qquad\text{for all }m\in\mathbb F_p.
\end{equation}

Next we compute $\|\widehat g\|_2$ exactly. Define
$$
q(x)=\phi(x(x-1)),
\qquad x\in\mathbb F_p.
$$
For $m\ne 0$, equation~\eqref{eq:ghatformula} and the identity $|\tau(\phi)|=\sqrt p$ give
$$
|\widehat g(m)|
=
p^{-1/2}\left|\sum_{x\in\mathbb F_p}q(x)\chi(-mx)\right|
=
|\widehat q(m)|.
$$
Since $\widehat g(0)=0$, Plancherel gives
$$
\|\widehat g\|_2^2
=
\sum_{m\ne 0}|\widehat q(m)|^2
=
\|q\|_2^2-|\widehat q(0)|^2.
$$
Now $|q(x)|=1$ for $x\notin\{0,1\}$ and $q(0)=q(1)=0$, so
$$
\|q\|_2^2=p-2.
$$
The standard quadratic character identity
$$
\sum_{x\in\mathbb F_p}\phi(x(x-1))=-1
$$
gives
$$
\widehat q(0)=-p^{-1/2}.
$$
Therefore
\begin{equation}\label{eq:gL2exact}
\|\widehat g\|_2^2=p-2-\frac{1}{p}.
\end{equation}
By Plancherel, the same identity holds for $\|g\|_2^2$. Since $f=g$ away from $0$ and $1$, while $f(0)=f(1)=0$, we obtain
\begin{align*}
\|f\|_2^2
&=\|g\|_2^2-|g(0)|^2-|g(1)|^2 \\
&=p-2-\frac{1}{p}-\frac{2}{p}
=p-2-\frac{3}{p}.
\end{align*}
Thus
\begin{equation}\label{eq:fL2exact}
\|\widehat f\|_2^2=p-2-\frac{3}{p}.
\end{equation}

We now estimate $\|\widehat f\|_1$. First, \eqref{eq:Linftybound} and the estimate $|\widehat h(m)|\le 2/p$ imply, after enlarging $C_0$ if necessary, that
$$
|\widehat f(m)|\le C_0
\qquad\text{for all }m\in\mathbb F_p.
$$
Consequently,
$$
\|\widehat f\|_2^2
=
\sum_{m\in\mathbb F_p}|\widehat f(m)|^2
\le
\|\widehat f\|_\infty\|\widehat f\|_1
\le
C_0\|\widehat f\|_1.
$$
Using \eqref{eq:fL2exact}, we conclude that
\begin{equation}\label{eq:L1lower}
\|\widehat f\|_1\gg p.
\end{equation}

For the upper bound, the Cauchy--Schwarz inequality gives
$$
\|\widehat f\|_1\le \sqrt p\,\|\widehat f\|_2.
$$
Since \eqref{eq:fL2exact} implies $\|\widehat f\|_2\asymp\sqrt p$, we obtain
\begin{equation}\label{eq:L1upper}
\|\widehat f\|_1\ll p.
\end{equation}

Combining \eqref{eq:L1lower} and \eqref{eq:L1upper}, we get
$$
\|\widehat f\|_1\asymp p,
\qquad
\|\widehat f\|_2\asymp\sqrt p.
$$
Therefore
$$
\frac{\|\widehat f\|_1}{\|\widehat f\|_2}\asymp\sqrt p.
$$
This completes the proof.
\end{proof}

\begin{remark}
The theorem shows that the Legendre family gives a function on $\mathbb F_p$ with large Fourier ratio:
$$
\|\widehat f\|_1\asymp p
\qquad\text{and}\qquad
\|\widehat f\|_2\asymp\sqrt p.
$$
Thus the Fourier mass of $f$ is spread across $\asymp p$ frequencies at essentially bounded size. An analogous question for the M\"obius function is studied in \cite{BursteinIosevichSant2026}. In that setting, one obtains an unconditional lower bound for the Fourier ratio, while a substantially stronger lower bound follows from a square-root cancellation hypothesis for the relevant exponential sums. For the Legendre family, the corresponding boundedness and $L^2$ estimates are available unconditionally.
\end{remark}

\section{Sparse approximation lower bounds via Fourier complexity}
\label{section:VC}

In this section we show that functions with large Fourier Ratio cannot be efficiently approximated by sparse Fourier models. The argument proceeds directly from Fourier-analytic considerations.

\subsection{Approximation by sparse Fourier models}

Let $g:\mathbb{F}_p\to\mathbb{C}$. For a subset $\Lambda\subset\mathbb{F}_p$, define the Fourier projection
$$
P_\Lambda g(t)=p^{-1/2}\sum_{m\in\Lambda}\widehat g(m)\chi(mt).
$$
Thus $P_\Lambda g$ is obtained from the Fourier inversion formula by retaining only the frequencies in $\Lambda$. Moreover, $P_\Lambda g$ is the best $L^2$ approximation to $g$ among all Fourier polynomials supported on $\Lambda$.

The following lemma records both the approximation error and this optimality property.

\begin{lemma}\label{lem:sparse-approx}
Let $g:\mathbb{F}_p\to\mathbb{C}$ and let $\Lambda\subset\mathbb{F}_p$. Then
$$
\|g-P_\Lambda g\|_2^2
=
\sum_{m\notin\Lambda}|\widehat g(m)|^2.
$$
Furthermore, if
$$
T(t)=p^{-1/2}\sum_{m\in\Lambda}c_m\chi(mt),
$$
then
$$
\|g-T\|_2^2
=
\|g-P_\Lambda g\|_2^2
+
\|P_\Lambda g-T\|_2^2.
$$
In particular,
$$
\|g-T\|_2\ge \|g-P_\Lambda g\|_2.
$$
\end{lemma}

\begin{proof}
The normalized Fourier inversion formula gives
$$
g(t)-P_\Lambda g(t)
=
p^{-1/2}\sum_{m\notin\Lambda}\widehat g(m)\chi(mt).
$$
Plancherel therefore yields
$$
\|g-P_\Lambda g\|_2^2
=
\sum_{m\notin\Lambda}|\widehat g(m)|^2.
$$
Since $g-P_\Lambda g$ and $P_\Lambda g-T$ have disjoint Fourier supports, they are orthogonal in $L^2(\mathbb F_p)$. Hence
$$
\|g-T\|_2^2
=
\|(g-P_\Lambda g)+(P_\Lambda g-T)\|_2^2
=
\|g-P_\Lambda g\|_2^2
+
\|P_\Lambda g-T\|_2^2.
$$
The final assertion follows immediately.
\end{proof}

\subsection{Consequence of large Fourier Ratio}

The argument below is consistent with a general principle developed elsewhere, namely that large Fourier Ratio forces high spectral and sampling complexity. We give a direct proof tailored to the present setting, which is consistent with the general framework developed in the literature.

We now show that large Fourier Ratio forces poor approximation by any sparse Fourier model. The key is the uniform bound $|\widehat{g}(m)| \le C_0$ from the Weil bound, which allows us to handle $k=O(p)$ rather than just $O(\sqrt p)$.

\begin{proposition}\label{prop:FR-hard}
Let $g:\mathbb F_p\to\mathbb C$ satisfy
$$
\|\widehat g\|_1\asymp p,
\qquad
\|\widehat g\|_\infty\le C_0,
\qquad
\|\widehat g\|_2\asymp\sqrt p.
$$
Then there exist absolute constants $c,c'>0$ such that, for every set $\Lambda\subset\mathbb F_p$ with $|\Lambda|\le c'p$ and every Fourier polynomial $T$ supported on $\Lambda$,
$$
\|g-T\|_2\ge c\|g\|_2.
$$
\end{proposition}

\begin{proof}
Choose absolute constants $A,B>0$ such that
$$
\|\widehat g\|_2^2\ge Ap
\qquad\text{and}\qquad
\|g\|_2=\|\widehat g\|_2\le B\sqrt p.
$$
For every $\Lambda\subset\mathbb F_p$, the uniform bound on the Fourier coefficients gives
$$
\sum_{m\in\Lambda}|\widehat g(m)|^2
\le
C_0^2|\Lambda|.
$$
Set
$$
c'=\frac{A}{2C_0^2}.
$$
If $|\Lambda|\le c'p$, then Lemma~\ref{lem:sparse-approx} gives
\begin{align*}
\|g-P_\Lambda g\|_2^2
&=\sum_{m\notin\Lambda}|\widehat g(m)|^2 \\
&=\|\widehat g\|_2^2-\sum_{m\in\Lambda}|\widehat g(m)|^2 \\
&\ge Ap-C_0^2|\Lambda| \\
&\ge \frac{A}{2}p.
\end{align*}
Since $\|g\|_2\le B\sqrt p$, it follows that
$$
\|g-P_\Lambda g\|_2
\ge
\frac{\sqrt{A}}{\sqrt{2}B}\|g\|_2.
$$
The optimality assertion in Lemma~\ref{lem:sparse-approx} now gives
$$
\|g-T\|_2\ge \|g-P_\Lambda g\|_2
$$
for every Fourier polynomial $T$ supported on $\Lambda$. This completes the proof.
\end{proof}

The argument above shows that the normalized Frobenius trace is globally difficult to approximate by sparse Fourier models. There is also a complementary manifestation of this complexity. Since
$$
\|\widehat f\|_2^2\asymp p
\qquad\text{and}\qquad
|\widehat f(m)|\le C_0,
$$
a positive proportion of the Fourier energy is carried by a set of $\asymp p$ frequencies on which the coefficients are bounded below by an absolute constant. By changing signs simultaneously on the paired frequencies $m$ and $-m$, one obtains an exponentially large family of real-valued functions with the same Fourier magnitude profile and the same Fourier ratio as $f$. These functions are pairwise separated by a constant multiple of $\sqrt p$. Thus a single Fourier magnitude profile is compatible with exponentially many well-separated signals.

An analogous phase-orbit packing argument, developed in a different context for the metric entropy of Fourier-ratio classes on $\mathbb Z_N$, can be found in \cite{IosevichHovhannisyanKeyshamsVagharshakyan2026}.

\begin{proposition}\label{prop:local-family}
There exist a symmetric set $\Lambda\subset\mathbb F_p$ with $|\Lambda|\asymp p$ and a family $\mathcal F_f$ of real-valued functions on $\mathbb F_p$ such that
$$
|\mathcal F_f|\ge\exp(cp)
$$
for some absolute constant $c>0$. Every $g\in\mathcal F_f$ satisfies
$$
|\widehat g(m)|=|\widehat f(m)|
\qquad\text{for all }m\in\mathbb F_p,
$$
and hence
$$
FR(g)=FR(f).
$$
Moreover, if $g,h\in\mathcal F_f$ are distinct, then
$$
\|g-h\|_2\ge c\sqrt p.
$$
\end{proposition}

\begin{proof}
From Theorem~\ref{thm:LegendreFourierRatio} and its proof, we have
$$
\sum_{m\in\mathbb F_p}|\widehat f(m)|^2\asymp p
\qquad\text{and}\qquad
|\widehat f(m)|\le C_0
$$
for every $m\in\mathbb F_p$. Since $f$ is real-valued,
$$
\widehat f(-m)=\overline{\widehat f(m)}.
$$

Fix a sufficiently small absolute constant $a>0$, and define
$$
\Lambda=\{m\in\mathbb F_p:|\widehat f(m)|\ge a\}.
$$
Then $\Lambda$ is symmetric. Moreover,
$$
\sum_{m\notin\Lambda}|\widehat f(m)|^2\le a^2p.
$$
Choosing $a$ sufficiently small gives
$$
\sum_{m\in\Lambda}|\widehat f(m)|^2\ge cp.
$$
Since $|\widehat f(m)|\le C_0$, it follows that
$$
|\Lambda|\ge cp.
$$
Thus $|\Lambda|\asymp p$.

Remove the frequency $0$ from $\Lambda$ if it is present, and choose a set $R\subset\mathbb F_p$ containing exactly one representative from each pair
$$
\{m,-m\}\subset\Lambda.
$$
Let $N=|R|$. Since $p$ is odd and $|\Lambda|\asymp p$, we have
$$
N\asymp p.
$$
After decreasing the absolute constants if necessary, the finitely many small primes are included as well.

We now choose a set
$$
\Sigma\subset\{-1,1\}^R
$$
such that
$$
|\Sigma|\ge\exp(cN)
$$
and any two distinct $\sigma,\tau\in\Sigma$ differ on at least $N/4$ elements of $R$. This follows from the standard greedy coding argument. Indeed, a Hamming ball of radius less than $N/4$ has cardinality at most
$$
\sum_{j<N/4}\binom Nj
\le
\exp(H(1/4)N),
$$
where
$$
H(x)=-x\log x-(1-x)\log(1-x).
$$
Since $H(1/4)<\log 2$, the greedy procedure produces at least
$$
\exp((\log 2-H(1/4))N)
$$
such sign patterns.

For each $\sigma\in\Sigma$, define $\varepsilon_\sigma:\mathbb F_p\to\{-1,1\}$ by
$$
\varepsilon_\sigma(m)=
\begin{cases}
\sigma(r), & m=r\text{ or }m=-r\text{ for some }r\in R,\\
1, & \text{otherwise}.
\end{cases}
$$
Define $f_\sigma$ by
$$
\widehat{f_\sigma}(m)
=
\varepsilon_\sigma(m)\widehat f(m).
$$
Since
$$
\varepsilon_\sigma(-m)=\varepsilon_\sigma(m),
$$
we have
$$
\widehat{f_\sigma}(-m)
=
\overline{\widehat{f_\sigma}(m)}.
$$
Thus every $f_\sigma$ is real-valued. Also,
$$
|\widehat{f_\sigma}(m)|=|\widehat f(m)|
$$
for every $m$, and therefore
$$
FR(f_\sigma)=FR(f).
$$

Set
$$
\mathcal F_f=\{f_\sigma:\sigma\in\Sigma\}.
$$
If $\sigma\ne\tau$, then by Plancherel,
$$
\|f_\sigma-f_\tau\|_2^2
=
\sum_{m\in\mathbb F_p}
|\widehat{f_\sigma}(m)-\widehat{f_\tau}(m)|^2.
$$
For each $r\in R$ on which $\sigma(r)\ne\tau(r)$, both frequencies $r$ and $-r$ contribute, and
$$
|\widehat{f_\sigma}(\pm r)-\widehat{f_\tau}(\pm r)|
=
2|\widehat f(r)|.
$$
Since $|\widehat f(r)|\ge a$ and $\sigma,\tau$ differ on at least $N/4$ elements of $R$, we obtain
$$
\|f_\sigma-f_\tau\|_2^2
\ge
\frac N4\cdot 8a^2
=
2a^2N
\ge
cp.
$$
Hence
$$
\|f_\sigma-f_\tau\|_2\ge c\sqrt p.
$$
This completes the proof.
\end{proof}

\section{Proof of the Main Result}
\label{section:proof}

We now assemble the pieces to prove Theorem~\ref{thm:main}.

\begin{proof}
The Fourier ratio statement follows from Theorem~\ref{thm:LegendreFourierRatio}.

The sparse approximation lower bound follows from Proposition~\ref{prop:FR-hard}, since Theorem~\ref{thm:LegendreFourierRatio} gives
$$
\|\widehat f\|_1 \asymp p,
\qquad
\|\widehat f\|_2 \asymp \sqrt p,
\qquad
|\widehat f(m)|\le C_0
$$
for all $m\in\mathbb F_p$.

The local complexity statement is exactly Proposition~\ref{prop:local-family}.

The sampling complexity statement is proved in Section~\ref{section:sampling}.
\end{proof}

\section{Sampling complexity consequences}
\label{section:sampling}

We now prove the sampling complexity statement in Theorem~\ref{thm:main}. Throughout this section, the reconstruction procedure is deterministic. It may choose its sampling points either nonadaptively or adaptively, but each observed value is encoded using at most $B$ bits.

Let $\mathcal F_f$ be the family constructed in Proposition~\ref{prop:local-family}. There are absolute constants $c_0,c_1>0$ such that
$$
|\mathcal F_f|\ge \exp(c_0p)
$$
and
$$
\|g-h\|_2\ge c_1\sqrt p
$$
whenever $g,h\in\mathcal F_f$ are distinct.

Suppose that a reconstruction procedure uses $m$ point evaluations and reconstructs every $g\in\mathcal F_f$ to error less than $c_1\sqrt p/3$ in $\ell^2$. Since each observed value is encoded using at most $B$ bits, each sample has at most $2^B$ possible encoded outcomes. Consequently, the total number of possible transcripts is at most
$$
2^{Bm}.
$$
This remains true for an adaptive procedure, because the location of each later sample is determined by the transcript obtained from the preceding samples. Equivalently, a deterministic adaptive procedure is represented by a decision tree of depth $m$ whose branching number is at most $2^B$.

Distinct functions in $\mathcal F_f$ must produce distinct transcripts. Indeed, if two distinct functions $g,h\in\mathcal F_f$ produced the same transcript, the procedure would return the same approximation $u$ for both functions. The triangle inequality would then give
$$
\|g-h\|_2
\le
\|g-u\|_2+\|h-u\|_2
<
\frac{2c_1\sqrt p}{3},
$$
contradicting
$$
\|g-h\|_2\ge c_1\sqrt p.
$$
Therefore
$$
\exp(c_0p)
\le
|\mathcal F_f|
\le
2^{Bm}.
$$
Taking logarithms gives
$$
m\ge \frac{c_0}{B\log 2}p.
$$
Thus
$$
m\ge c_Bp,
$$
where $c_B=c_0/(B\log 2)>0$ depends only on $B$. This proves the sampling complexity statement in Theorem~\ref{thm:main}.

\section{Conclusion and Open Problems}
\label{section:conclusion}

We have shown that the normalized Frobenius trace for the Legendre family of elliptic curves has Fourier ratio comparable to $\sqrt p$. Consequently, any sparse Fourier approximation using fewer than a sufficiently small constant multiple of $p$ frequencies incurs $L^2$ error at least a fixed fraction of $\|f\|_2$. We have also shown that the Fourier magnitude profile of $f$ generates an exponentially large family of pairwise separated signals with identical Fourier magnitudes and identical Fourier ratio. This yields a linear lower bound for uniform reconstruction of that family from bounded-precision point evaluations.

The arithmetic input is unconditional and relies only on the Weil bound for mixed character sums, the evaluation of the quadratic Gauss sum, and elementary character identities. The approximation and sampling consequences then follow from Plancherel, orthogonality, a coding argument, and transcript counting.

Several natural questions remain:

\begin{enumerate}

\item \textbf{Other families:} It is natural to ask which other one-parameter families of elliptic curves give rise to the same Fourier-ratio phenomenon. A first test case is
$$
E_t:\ y^2=x^3+x+t,
$$
with the parameters for which the discriminant vanishes omitted. More generally, one would like conditions on a family of elliptic curves, or on an associated family of $\ell$-adic sheaves, that guarantee uniformly bounded Fourier coefficients together with an $L^2$ norm of order $\sqrt p$. These two properties would imply a Fourier ratio of order $\sqrt p$ by the argument used in this paper. The trace-function framework suggests that the uniform Fourier bound should be available for broad bounded-conductor families, but the required nondegenerate second moment must also be verified.

\item \textbf{Arithmetic signal classes:} The family $\mathcal F_f$ is obtained by changing Fourier signs and need not consist of Frobenius trace functions. It would be considerably stronger to construct an exponentially large, pairwise separated class of actual arithmetic signals, such as Frobenius traces from a natural family or a family of quadratic twists, with a common or tightly controlled Fourier magnitude profile. Such a result would turn the present profile-based sampling obstruction into a genuinely arithmetic sampling theorem.

\item \textbf{Optimal constants:} Our proof gives the existence of constants $c,c'>0$ but does not optimize them. It would be interesting to determine the best possible constants for which the approximation lower bound holds. More detailed information about the distribution of the values $|\widehat f(m)|$ may lead to substantially sharper estimates than those obtained from the uniform bound and the second moment alone.

\item \textbf{Randomized and noisy reconstruction:} The sampling lower bound concerns deterministic reconstruction from bounded-precision point evaluations. It is natural to ask for analogous lower bounds for randomized algorithms, noisy observations, average-case recovery, or more general linear measurements. The large separated family constructed here is a natural starting point, but additional information-theoretic input would be needed.

\item \textbf{Connection to learning theory:} A meaningful learning-theoretic statement requires a hypothesis class rather than a single known function. The class $\mathcal F_f$ provides a natural metric-complexity obstruction, but it is not an arithmetic class. It would be interesting to identify a natural class of arithmetic trace functions for which one can prove VC-dimension, statistical-query, or distribution-independent sample-complexity lower bounds.

\item \textbf{Comparison with the M\"obius function:} In \cite{BursteinIosevichSant2026}, an unconditional lower bound for the Fourier ratio of the M\"obius function is obtained using recent estimates for its exponential sums. Under the stronger conjectural hypothesis
$$
\sup_{\theta\in\mathbb R}
\left|
\sum_{n\le R}\mu(n)e^{2\pi i n\theta}
\right|
\le
R^{1/2+o(1)},
$$
one obtains the substantially stronger conclusion
$$
FR(\mu_R)\ge R^{-o(1)}.
$$
The normalization in that continuous setting differs from the counting-norm normalization used here. For the Legendre family, the analogous boundedness of the Fourier coefficients and the required $L^2$ estimate are available unconditionally. It would be interesting to understand systematically which arithmetic families admit such unconditional Fourier-ratio bounds.

\end{enumerate}

\end{document}